\documentclass[1p]{elsarticle}
\usepackage[english]{babel}
\usepackage{amsmath,amsthm}
\usepackage{amsfonts}
\usepackage{amssymb}
\usepackage[prependcaption,textsize=scriptsize]{todonotes}
\usepackage{hyperref}
\usepackage{mathtools}
\usepackage[ruled, linesnumbered]{algorithm2e}
\usepackage{tikz}
\usetikzlibrary{shapes,arrows,backgrounds,fit,positioning}
\usepackage{graphicx}
\usepackage{caption}
\usepackage{subcaption}
\usepackage{threeparttable}
\usepackage{booktabs}
\usepackage{enumerate}
\usepackage{amsfonts}
\usepackage{faktor}
\usepackage{numprint}
\usepackage{listings}

\definecolor{ForestGreen}{RGB}{34,139,34}

\lstdefinelanguage{Julia}%
  {morekeywords={abstract,break,case,catch,const,continue,do,else,elseif,%
      end,export,false,for,function,immutable,import,importall,if,in,%
      macro,module,otherwise,quote,return,switch,true,try,type,typealias,%
      using,while},%
   sensitive=true,%
   alsoother={$},%
   morecomment=[l]\#,%
   morecomment=[n]{\#=}{=\#},%
   morestring=[s]{"}{"},%
   morestring=[m]{'}{'},%
}[keywords,comments,strings]%

\usetikzlibrary{calc}

\usetikzlibrary{arrows,matrix,positioning, decorations.pathreplacing,calc}
\newtheorem{thm}{Theorem}[section]
\newtheorem{cor}[thm]{Corollary}

\newtheorem{prop}[thm]{Proposition}
\theoremstyle{definition}
\newtheorem{defn}[thm]{Definition}
\theoremstyle{remark}
\newtheorem{rem}[thm]{Remark}
\numberwithin{equation}{section}

\newcommand{\ER}{Erd\H{o}s-R\'{e}nyi }

\usepackage{array}
\newcolumntype{H}{>{\setbox0=\hbox\bgroup}c<{\egroup}@{}}

\def\svec{\mbox{\bf svec}}

\begin{document}

\title{Jordan symmetry reduction for conic optimization over the doubly nonnegative cone: theory and software\tnoteref{n1}}
\tnotetext[n1]{
\begin{minipage}[b]{0.8\linewidth}
This project has received funding from the European Union’s Horizon 2020 research and innovation programme under the Marie Skłodowska-Curie grant agreement No 764759.
\end{minipage}
\hfill
\begin{minipage}[b]{0.15\linewidth}
\hfill\includegraphics[height=2.7\baselineskip]{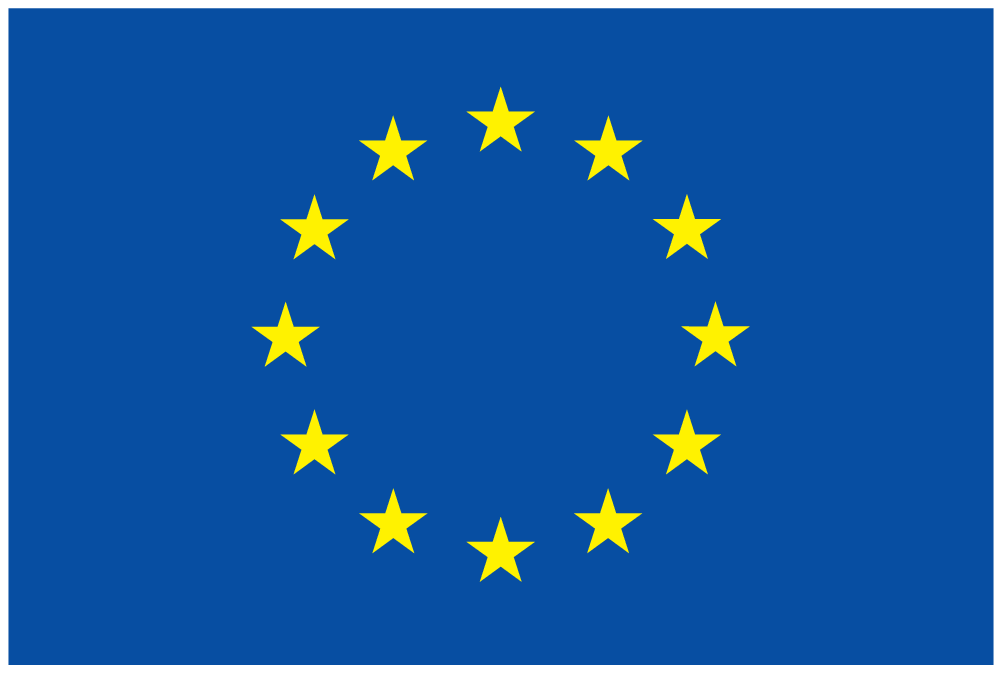}
\end{minipage}
}
%\title{A new look at symmetry reduction of semidefinite relaxations of the quadratic assignment problem \\ \emph{SHORT VERSION}}%
\author{Daniel Brosch, Etienne de Klerk}%
\address{Tilburg University}
%\address{}%
%\thanks{}%
\date{\today}%
% ----------------------------------------------------------------
\begin{abstract}
A common computational approach for polynomial optimization problems (POPs) is to use (hierarchies of) semidefinite programming (SDP) relaxations. When the variables in the POP are required to be nonnegative, these SDP problems typically involve nonnegative matrices, i.e. they are conic optimization problems over the doubly nonnegative cone. The Jordan reduction, a symmetry reduction method for conic optimization, was recently introduced for symmetric cones by Parrilo and Permenter [Mathematical Programming 181(1), 2020]. We extend this method to the doubly nonnegative cone, and investigate its application to known relaxations of the quadratic assignment and maximum stable set problems. We also introduce new Julia software where the symmetry reduction is implemented.
\end{abstract}
\begin{keyword}
Quadratic assignment problem \sep Semidefinite programming \sep Symmetry reduction \MSC{90C22\sep 20B40}
\end{keyword}
\maketitle
% ----------------------------------------------------------------

\section{Introduction}
This paper studies symmetry reduction of semidefinite programs (SDPs) where the matrix variable is also entry-wise nonnegative,
i.e.\ symmetry reduction of conic linear programming over the doubly nonnegative cone.
Such problems appear naturally in the study of convex relaxations of combinatorial problems.
In particular, we are interested in such relaxations of the independence number $\alpha$ of a graph, and the quadratic assignment problem (QAP).

The independence number $\alpha$ of a graph $G$ is the maximum number of nodes of $G$ we can choose, such that there is no edge between any of them. The Theta-Prime function ($\vartheta'$-function) is a semidefinite relaxation of $\alpha$, and as such gives an upper bound to it. Given the adjacency matrix $A$ of $G$ we have
\begin{align}\label{thetaPrime}
  \vartheta'(G) = \sup \enspace& \langle J,X\rangle \\
  \text{s.t.} \enspace& \mathrm{trace}(X) = 1, \nonumber\\
   & \langle A,X\rangle = 0,\nonumber\\
   & X \geq 0,\nonumber\\
   & X \succcurlyeq 0.\nonumber
\end{align}

The second family of problems we are interested are quadratic assignment problem, which are of the form
\begin{equation}
  QAP(A,B) = \min_{\varphi\in S_n} \sum_{i,j=1}^{n}a_{ij}b_{\varphi(i)\varphi(j)},
  \label{def:qapAB}
\end{equation}
where $A = (a_{ij})$ and $B = (b_{ij})$ are square $n\times n$ matrices, and $S_n$ denotes the symmetric group (i.e.\ all permutations) on $n$ elements. Here we are interested in the SDP relaxation of by Zhao et al.\ \cite{zhao1998semidefinite}.

Symmetry reduction for SDP was first introduced by Schrijver in 1979 in \cite{schrijver1979comparison};
 see for example the chapter \cite{bachoc2012invariant}
  by Bachoc, Gijswijt, Schrijver and Vallentin for a review of later developments up to 2012. The specific case of SDP relaxations of quadratic assignment problems was investigated
   by De Klerk and Sotirov in \cite{de2010exploiting,de2012improved}.

Parrilo and Permenter \cite{permenter2016dimension} recently introduced a new --- and more general --- form of symmetry reduction, called \emph{Jordan reduction}.
We will extend their approach to the doubly-nonnegative cone.

\subsection*{Outline and contributions of this paper}
In the next section, we recap relevant definitions and results on the Jordan reduction of Parrilo and Permenter \cite{permenter2016dimension}.
In Section \ref{sec:extenD} we subsequently extend this approach --- which was formulated for symmetric cones --- to the doubly nonnegative cone.
This allows us to apply the method to the SDP relaxation of the general QAP due to  Zhao et al.\ \cite{zhao1998semidefinite} in Section \ref{sec:QAPreduce}, and to the $\vartheta'$-function of \ER-graphs in Section \ref{ThetaER}.
Our extension of the Jordan reduction method of Parrilo and Permenter \cite{permenter2016dimension} in Section \ref{sec:extenD} should lead to additional applications in SDP relaxations of other combinatorial problems. Finally, in Section \ref{JuliaPackageSection}, we describe a Julia software package implementing this method.

\section{Preliminaries on Jordan symmetry reduction }
We will study conic optimization problems in the form
\begin{equation}\label{def:conic opt problem}
\left.
\begin{array}{llrl}
  \inf\enspace&\langle C,X\rangle & \;\;=\inf\enspace&\langle C,X\rangle  \\
  \mathrm{s.t.}\enspace& \langle A_i,X\rangle = b_i \text{ for } i \in [m] &\;\;\mathrm{s.t.}\enspace& X\in X_0+\mathcal{L} \\
   &X\in \mathcal{K}&& X\in \mathcal{K},
\end{array}
  \right\}
  \end{equation}
where $[m] = \{1,\ldots,m\}$, $\mathcal{K}\subseteq \mathcal{V}$ is a closed, convex cone in a real Hilbert space $\mathcal{V}$, $X_0 \in \mathcal{V}$ satisfies
 $\langle A_i, X_0 \rangle = b_i$ for all $i \in [m]$, and $\mathcal{L}\subseteq \mathcal{V}$ is the nullspace of the linear operator $A$,
 where $A(X) = \left(\langle A_i,X\rangle\right)_{i=1}^m$. The objective function is given using the inner product $\langle \cdot,\cdot\rangle$ of $\mathcal{V}$, with which one defines the \emph{dual cone} as:
\begin{equation*}
  \mathcal{K}^*\coloneqq \{Y\in \mathcal{V}\enspace\vert\enspace \langle X,Y\rangle\geq 0\enspace\forall X\in\mathcal{K}\}.
\end{equation*}
In this paper, we will mostly deal with the case where $\mathcal{V}$ is the space $\mathbb{S}^n$ of $n\times n$ symmetric matrices equipped with the Euclidean inner product, and where $\mathcal{K}$ is the cone of doubly nonnegative matrices.

\subsection{Constraint Set Invariance}
Parrilo and Permenter \cite{permenter2016dimension} introduced a set of three conditions a subspace has to fulfill,
such that it is be possible to use it for symmetry reduction. Here we  revisit some of their results.
\begin{defn}
  A \emph{projection} is a linear transformation $P\colon \mathcal{V}\to\mathcal{V}$ which is \emph{idempotent}, i.e. $P^2=P$.
\end{defn}

\begin{defn}[Definition 2.1. in \cite{permenter2016dimension}] A projection $P\colon \mathcal{V}\to\mathcal{V}$ fulfills the \emph{Constraint Set Invariance Conditions (CSICs)} for $(\mathcal{K},X_0+\mathcal{L},C)$ if
\begin{itemize}
  \item[(i)] The projection is positive: $P(\mathcal{K})\subseteq \mathcal{K}$,
  \item[(ii)] $P(X_0+\mathcal{L})\subseteq X_0+\mathcal{L}$,
  \item[(iii)] $P^*(C+\mathcal{L}^\perp)\subseteq C+\mathcal{L}^\perp$,
\end{itemize}
where $P^*$ is the adjoint of $P$, which satisfies $\langle P(X),Y\rangle = \langle X, P^*(Y)\rangle$ for all $X,Y\in \mathcal{V}$.
\end{defn}

%Note that this definition is symmetric in the sense that it does not change if we switch the primal and dual program, since for $X\in \mathcal{K}$ and $Y\in \mathcal{K}^*$ and positive $P$ we have
%\begin{equation*}
%  0\leq \langle P(X),Y\rangle = \langle X,P^*(Y)\rangle.
%\end{equation*}
%Thus
Note that this definition is symmetric going from primal to dual, since
\begin{equation*}
  P(\mathcal{K})\subseteq \mathcal{K} \quad\Leftrightarrow\quad P^*(\mathcal{K}^*)\subseteq \mathcal{K}^*.
\end{equation*}

These projections send feasible solutions to feasible solutions with the same objective value, as the next result shows.
%(We include a proof for completeness.)
\begin{prop}[Proposition 1.4.1 in \cite{permenter2017reduction}]\label{PropertiesCSIC}
  If a projection $P\colon \mathcal{V}\to\mathcal{V}$ fulfills the CSICs, then
  \begin{itemize}
    \item $P((X_0+\mathcal{L})\cap \mathcal{K})\subseteq (X_0+\mathcal{L})\cap\mathcal{K}$,
    \item $P^*((C+\mathcal{L}^\perp)\cap\mathcal{K}^*)\subseteq (C+\mathcal{L}^\perp)\cap\mathcal{K}^*$,
    \item For $X\in X_0+\mathcal{L}$:\quad $\langle C,X\rangle=\langle C,P(X)\rangle$,
    \item For $Y\in C+\mathcal{L}^\perp$:\quad$\langle X_0, Y\rangle=\langle X_0, P^*(Y)\rangle$.
  \end{itemize}
\end{prop}
% Proof just in thesis
%\begin{proof}[Proof from \cite{permenter2017reduction}]
%    The first two properties are direct consequences of the definition, and the remaining two are equivalent by its symmetry. Note that because of \emph{(ii)} the difference $X-P(X)$ lies in $\mathcal{L}$ for any $X\in X_0+\mathcal{L}$, and similarly $C-P^*(C)\in\mathcal{L}^\perp$, which means that these vectors are orthogonal.
%    \begin{align*}
%      \langle C,X\rangle &= \langle C-P^*(C)+P^*(C),X-P(X)+P(X)\rangle \\
%      &= \underbrace{\langle C-P^*(C), X-P(X)\rangle}_{=0} +  \langle C-P^*(C), P(X)\rangle \\&\quad+ \langle P^*(C),X\rangle + \underbrace{\langle P^*(C), P(X)-P(X)\rangle}_{=0}\\
%      &= \langle P^*(C)-P^*(C), X\rangle + \langle C,P(X)\rangle\\
%      &= \langle C,P(X)\rangle
%    \end{align*}
%    Here we used $PP=P$ and $\langle P(A),B\rangle=\langle A,P^*(B)\rangle$ in the penultimate step.
%\end{proof}

To make things easier, we restrict ourselves to \emph{orthogonal} projections $P_S$ to a subspace $S\subseteq\mathcal{V}$,
 which are exactly the projections of which the range and kernel are orthogonal,
 or equivalently the projections which are self-adjoint, i.e. $P_S=P_S^*$.
  If the projection $P_S$ fulfills the CSICs we call the subspace $S$ \emph{admissible}, following \cite{permenter2016dimension}.
In this case, the CSICs may be rewritten as follows.

\begin{thm}[Theorem 5.2.4 in \cite{permenter2017reduction}]
\label{thm:equivalent CSICs}
Consider the conic optimization problem \eqref{def:conic opt problem} and let $S \subseteq \mathcal{V}$ be the range
of an orthogonal projection $P_S:\mathcal{V} \rightarrow \mathcal{V}$.
Let $P_{\mathcal{L}}$ denote the orthogonal projection onto $\mathcal{L}$, etc., and define
$
C_{\mathcal{L}} = P_{\mathcal{L}}(C)
$
and $X_{0, \mathcal{L}^\perp} = P_{\mathcal{L}^\perp}(X_0)$.
Then $S$ is an admissible subspace if, and only if,
 \begin{itemize}
      \item[(a)] $C_\mathcal{L},X_{0,\mathcal{L}^\perp} \in S$,
     \item[(b)] $P_\mathcal{L}(S) \subseteq S$,
      \item[(c)]$P_S(\mathcal{K}) \subseteq \mathcal{K}$.
 \end{itemize}
\end{thm}

Restricting the conic program to an admissible subspace $S$ thus results in another, potentially significantly smaller program, with the same optimal value.
\begin{align*}
 \inf\enspace&\langle P_S(C),X\rangle \\
  \mathrm{s.t.}\enspace & X\in P_S(X_0)+\mathcal{L}\cap S, \\
  & X\in \mathcal{K}\cap S.
\end{align*}

\subsection{The reduction for Jordan-Algebras}
Next we review some results from \cite{permenter2017reduction} for the situation where the  space $\mathcal{V}$ is an \emph{Euclidian Jordan algebra} $\mathcal{J}$,
that is a commutative algebra (with product denoted by `$\circ$') over $\mathbb{R}$ satisfying the \emph{Jordan identity}
\begin{equation*}
  (x\circ y)\circ x^2=x\circ (y\circ x^2),
\end{equation*}
and an inner product with $\langle x\circ y,z\rangle=\langle y,x\circ z\rangle$.
For every such algebra we can define $\mathcal{K}$ as the cone of squares of $\mathcal{J}$ given by $\mathcal{K}=\{x\circ x\enspace\vert\enspace x\in \mathcal{J}\}$,
which always is a symmetric cone, i.e. a self-dual and homogenous convex cone (see for example \cite{Faraut1994AnalysisOS}).

The only example relevant for us is the case $\mathcal{J}={\mathbb{S}}^n$, the  symmetric $n\times n$-matrices with real entries, with product defined by
\begin{equation*}
  X\circ Y \coloneqq \frac{1}{2}(XY+YX),
\end{equation*}
and the inner product the Euclidean (trace) inner product $\langle X,Y \rangle = \mbox{trace}(XY)$.
It is easy to see (e.g. from the spectral decomposition) that its cone of squares is exactly the positive semidefinite cone ${\mathbb{S}}^n_+$.

Since the product of a Jordan algebra is commutative, we have
\begin{equation*}
  2x\circ y = x\circ y + y \circ x = (x+y)^2 - x^2 - y^2,
\end{equation*}
which means that subspaces are closed under multiplication, if and only if they include all squares. Similarly isomorphisms between (euclidian) Jordan algebras are exactly the bijective linear maps satisfying $\phi(x^2)=(\phi(x))^2$.

\begin{defn}
A Jordan algebra $\mathcal{J}$ is called \emph{special}, if it is isomorphic to  the algebra one gets from a real associative algebra  by equipping the latter with the product $x \circ y = \frac{1}{2}(xy+yx)$.
\end{defn}
There is only a single (up to isomorphisms) simple Jordan algebra which is not special, the algebra of Hermitian $3\times 3$-matrices of Octonions $H_3(\mathbb{O})$. The for us relevant case $\mathcal{J}={\mathbb{S}}^n$ is special.

\begin{defn}A subspace (not necessarily a subalgebra) $S$ of a Jordan algebra is called \emph{unital}, if there is an an element $e\in S$ such that $e\circ a=a\circ e=a$ for all $a\in S$.
\end{defn}
An important fact for us is that every Euclidean Jordan algebra is unital.

One main result of \cite{permenter2017reduction} is an alternative description of the CSICs when the ambient space is a special Euclidean Jordan algebra.
In this case the condition $P_S(\mathcal{K}) \subseteq \mathcal{K}$ in Theorem \ref{thm:equivalent CSICs} --- with $\mathcal{K}$ the
cone of squares in $\mathcal{\mathcal{J}}$ --- is equivalent to $S$ being closed under taking squares,
 i.e.\ to $S$ being a Jordan sub-algebra
of $\mathcal{\mathcal{J}}$.

This gives an algorithm for finding the minimal admissible subspace, which is defined as follows.
%Since the intersection of two admissible (unital) subspaces is again admissible (and unital), we can define the unique minimal (lowest dimension) subspace satisfying the CSICs.
\begin{defn}The unique minimal admissible subspace is
  \begin{equation*}
    S_{\min}\coloneqq \bigcap_{S \text{ is admissible}} S.
  \end{equation*}
%  If $\mathcal{V}=\mathcal{J}$ is an euclidian Jordan algebra, we can therefore define the minimal unital admissible subspace as
%  \begin{equation*}
%    S_{\min,\mathrm{unit}}\coloneqq \bigcap_{S \text{ is admissible and unital}} S.
%  \end{equation*}
\end{defn}

As mentioned before, we may now formulate an algorithm for $S_{\min}$. %and $S_{\min,\mathrm{unit}}$.
\begin{thm}[Theorem 3.2 in \cite{permenter2016dimension}]
  If $\mathcal{V}=\mathcal{J}$ is an Euclidian, special Jordan algebra, then $S_{\min}$ is the output of Algorithm \ref{AlgoSMinSpecial}.

  \begin{algorithm}[H]
  \caption{Finding $S_{\min}$}
  \label{AlgoSMinSpecial}
    $S \gets \mathrm{span}\{C_\mathcal{L},X_{0,\mathcal{L}^\perp}\}$\\
    \Repeat{converged}{
        $S\gets S+P_{\mathcal{L}}(S)$\\
        $S\gets S+\mathrm{span}\{X^2\enspace\vert\enspace X\in S\}$
    }
  \end{algorithm}
\end{thm}

 \subsection{A combinatorial reduction algorithm}
  \label{sec:combreduc}
    The fourth step of Algorithm \ref{AlgoSMinSpecial} is not linear, and  hard to implement.
     But, conveniently, Permenter does introduce three combinatorial algorithms in his PhD thesis (\cite{permenter2017reduction}, Chapter 7) for the cone ${\mathbb{S}}^n_+$, which all find orthogonal $0/1$-bases for an optimal unital admissible subspace with certain additional properties.  Here we will only mention one of the algorithms, since the other ones cannot give us better reductions for our special case.

\subsubsection*{Partition subspaces}
    The second combinatorial algorithm by Permenter \cite{permenter2017reduction} finds an optimal unital \emph{partition subspace}, which is a subspace with $0/1$-basis,
    the elements of which sum to the all-one matrix. We can describe the basis uniquely with a partition of the coordinates of ${\mathbb{S}}^n$, i.e.\ of $[n] \times [n]$,
     simply by having one part for every basis element with ones in the corresponding coordinates. For example the following spaces are partition spaces:
    \begin{align*}
      P_1 &= \begin{pmatrix}
               a & a & b \\
               a & a & b \\
               b & b & c \\
             \end{pmatrix},
    & P_2 &= \begin{pmatrix}
               a & b & b \\
               b & a & b \\
               b & b & c \\
             \end{pmatrix},
    & P_3 = P_1 \wedge P_2 &= \begin{pmatrix}
               a & b & c \\
               b & a & c \\
               c & c & d \\
             \end{pmatrix},
    \end{align*}
    where $P_1 \wedge P_2$ is the coarsest partition space \emph{refining} both $P_1$ and $P_2$.

    For our purposes an important special case is a so-called Jordan configuration, defined as follows.

    \begin{defn}
      A partition $P$ of $A\times A$, where $A$ is a finite set, is called a \emph{Jordan configuration}, if its characteristic matrices $\mathcal{B}_P$ satisfy
      \begin{itemize}
        \item $X=X^T$ for all $X\in \mathcal{B}_P$,
        \item $XY+YX\in \mathrm{span} \enspace\mathcal{B}_P$ for all $X,Y\in \mathcal{B}_P$,
        \item $I\in \mathrm{span} \enspace\mathcal{B}_P$.
      \end{itemize}
    \end{defn}
   In words, a Jordan configuration is a basis of a unital partition space that is also a Jordan subalgebra of ${\mathbb{S}}^n$.

A more general example of a partition space, also of interest to us, is a so-called coherent algebra.
 \begin{defn}
      A partition $P$ of $A\times A$, where $A$ is a finite set, is called a \emph{coherent configuration}, if its characteristic matrices $\mathcal{B}_P$ satisfy
      \begin{itemize}
        \item If $X\in \mathcal{B}_P$ then also $X^T\in \mathcal{B}_P$,
        \item $XY \in \mathrm{span} \enspace\mathcal{B}_P$ for all $X,Y\in \mathcal{B}_P$,
        \item $I\in \mathrm{span} \enspace\mathcal{B}_P$.
      \end{itemize}
    \end{defn}
    Thus, a coherent configuration gives a 0/1 basis of a partition subspace that is also a matrix $*$-algebra, namely the associated coherent algebra.
Note that the symmetric part of a coherent configuration is a Jordan configuration. It is an open question if the converse is also true \cite{Cameron}, \cite[p. 218]{permenter2017reduction}.

    To restrict the algorithm \ref{AlgoSMinSpecial} to partition subspaces, we need more notation: $\mathrm{part}(A)$ is the smallest
     partition space containing the matrix (or subspace) $A$, which is simply the partition space given by the unique entries of $A$.

    \begin{algorithm}
    \caption{Partition algorithm}
    $P \gets \mathrm{part}(C_\mathcal{L})\wedge \mathrm{part}(X_{0,\mathcal{L}^\perp})$\\
    \Repeat{converged}{
        $P\gets P\wedge\mathrm{part}(P_{\mathcal{L}}(P))$\\
        $P\gets P\wedge\mathrm{part}(\mathrm{span}\{X^2\enspace\vert\enspace X\in P\})$
    }
    \end{algorithm}

    There are two basic ways to implement this algorithm: One can use polynomial matrices, or randomization.
    For the first variant one introduces (commuting) variables $t_i$ for each element of a basis $B_1,\ldots, B_k$ of $P$,
    and then refines the partition with $\mathrm{part}(P_{\mathcal{L}}(\sum_{i=1}^{k}t_iB_i))= \mathrm{part}(\sum_{i=1}^{k}t_iP_{\mathcal{L}}(B_i))$ and
     $\mathrm{part}((\sum_{i=1}^{k}t_iB_i)^2)$. If we for example take $P_2$ from the example above, one has
     \begin{align*}
      \begin{pmatrix}
               t_a & t_b & t_b \\
               t_b & t_a & t_b \\
               t_b & t_b & t_c \\
             \end{pmatrix}^2
       = \begin{pmatrix}
               t_a^2+2t_b^2 & 2t_at_b+t_b^2 & t_at_b+t_b^2+t_bt_c \\
               2t_at_b+t_b^2 & t_a^2+2t_b^2 & t_at_b+t_b^2+t_bt_c \\
               t_at_b+t_b^2+t_bt_c & t_at_b+t_b^2+t_bt_c & 2t_b^2+t_c^2 \\
             \end{pmatrix},
    \end{align*}
    of which the unique polynomials induce the partition $P_3$.

    The second variant refines the partition with a random element in the partition space after projecting it to $\mathcal{L}$ and after squaring it.
    While one has to be more careful about rounding errors here, it is both easier to implement and much faster.
  %   The chance to hit a linear combination which does not change the partition, if it has not converged yet,
%     is near zero assuming finite precision, so one can stop the algorithm as soon as the partition does not change in an iteration:

    \begin{algorithm}
        \caption{Partition algorithm, randomized \label{alg:Partitionrandomized}}
    $P \gets \mathrm{part}(C_\mathcal{L})\wedge \mathrm{part}(X_{0,\mathcal{L}^\perp})$\\
    \Repeat{converged}{
        $X\gets \text{random element of } P$\\
        $P\gets P\wedge\mathrm{part}(P_{\mathcal{L}}(X))$\\
        $P\gets P\wedge\mathrm{part}(X^2)$
    }
    \end{algorithm}

    %\begin{rem}
%      A couple of specifics about the implementation: An easy way to represent a partition space of dimension $n$ is by describing it using a single matrix with elements in $1,\ldots,n$. That way one can simply apply a function mapping the unique elements of $X$ to $1,\ldots,n$ (e.g. a dictionary) to get $\mathrm{part}(X)$. If one has two partition spaces (in single matrix form) $P$ and $Q$ of dimensions $n$ and $m$ respectively, it is easy to see that
%      \begin{equation*}
%        P\wedge Q = \mathrm{part}((m+1)P+Q) = \mathrm{part}(P+(n+1)Q),
%      \end{equation*}
%      since then the entries of $(m+1)P+Q$ are uniquely determined by the parts of $P$ and $Q$ containing it.
%    \end{rem}

\begin{rem}
We note that the first variant of the partition algorithm presented here is very similar to the Weisfeiler-Leman (WL) algorithm \cite{WL}, that finds the coarsest
 coherent configuration
refining a given partition of $[n] \times [n]$. The only difference is that the  WL algorithm uses non-commuting variables $t_i$, as opposed to commuting ones; see \cite{Babel_et_al_WL}
on details of the implementation of the WL algorithm.
\end{rem}

\section{Extension to the doubly nonnegative cone}
\label{sec:extenD}
We will now fix the cone $\mathcal{K}$ in \eqref{def:conic opt problem} to be the \emph{doubly nonnegative cone} ${\mathcal{D}}^n\coloneqq {\mathbb{S}}^n_+ \cap \mathbb{R}^{n\times n}_+$.
Since we will refer to nonnegative, symmetric matrices frequently, we also introduce the notation $\mathcal{N}^n = {\mathbb{S}}^n \cap \mathbb{R}^{n\times n}_+$.
Even though ${\mathcal{D}}^n$ is not a cone of squares in a Euclidean Jordan algebra, one may
readily adapt some of the results of the last section to this setting.

We start with an elementary, but important observation.
\begin{prop}\label{SnPositiveToDn}
  Assume that a subspace
  $S \subset \mathbb{S}^n$
  has a basis of nonnegative matrices with pairwise disjoint supports.
   Then the orthogonal projection $P_S$ onto $S$ satisfies
     $P_S({\mathcal{D}}^n)\subseteq {D}^n$ if it satisfies
     $P_S({\mathbb{S}}^n_+)\subseteq {\mathbb{S}}^n_+$.
    \end{prop}
\begin{proof}
If $S$ has a basis of nonnegative matrices with disjoint supports, then it has an orthonormal basis with this property, say $A_i$ $(i\in [d])$, and the orthogonal projection is of the form
  \begin{equation*}
    P_S(X)=\sum_{i=1}^d \langle A_i,X\rangle A_i.
  \end{equation*}
   Since the Euclidean inner product of two nonnegative matrices is nonnegative, we have
  \begin{equation*}
    P_S(\mathcal{N}^n)\subseteq \mathcal{N}^n,
  \end{equation*}
  and, since $D_n \subset \mathbb{S}^n_+$, and  $P_S({\mathbb{S}}^n_+)\subseteq {\mathbb{S}}^n_+$ by assumption,
  \begin{equation*}
    P_S({\mathcal{D}}^n) \subseteq {\mathbb{S}}^n_+\cap \mathcal{N}^n={\mathcal{D}}^n.
  \end{equation*}
\end{proof}

If we consider partition subspaces, we may therefore use results on admissible partition subspaces for the case $\mathcal{K} = \mathbb{S}^n_+$, as follows.

\begin{cor}
Consider a conic optimization problem of the form \eqref{def:conic opt problem}, with $\mathcal{V }= \mathbb{S}^n$, and $\mathcal{K} = \mathbb{S}^n_+$, and let $S$ be a admissible partition subspace for this problem.
Then, $S$ is also an admissible partition subspace for the related problem where we replace $\mathcal{K} = \mathbb{S}^n_+$ by $\mathcal{K} = {\mathcal{D}}^n$.
\end{cor}

The important practical implication is that we may use Algorithm \ref{alg:Partitionrandomized} to find the minimal admissible Jordan configuration
for conic optimization problems on the cone ${\mathcal{D}}^n$. In the next section we will do precisely this for an SDP relaxation of the quadratic assignment problem.

It it instructive though, to ask how restrictive it is to only consider admissible partition subspaces.
In what follows, we show that, the partition subspace structure  is actually imposed by some relatively weak assumptions.

 To this end, we  recall
a result on nonnegative projection matrices, taken from \cite[Theorem 2.38]{projectors book}, but originally due to Belitskii and Lyubich (cf.\ \cite[p. 108]{Matrix Norms and their Applications}).

\begin{prop}[Theorem 2.1.11 in \cite{Matrix Norms and their Applications}]
\label{thm:nonnegative projection}
The general form of a nonnegative projection matrix is
\begin{equation}
\label{nonneg projection matrix}
  P = (A+B)C^T
\end{equation}
where $r = \mbox{rank}(P)$, $A,B,C \in \mathbb{R}_+^{n\times r}$, $A^TA = I$, $C^TA = I$, $B^TA = 0$ and $B^TC = 0$.
\end{prop}
As a consequence, a nonnegative projection matrix has the following structure.

\begin{cor}
\label{cor:nonnegative projection}
Any $n\times n$ symmetric nonnegative projection matrix $P$ with $r$-dimensional range takes the form $P = CC^T$ for some $C \in \mathbb{R}_+^{n\times r}$ such that $C^TC = I$.
In particular, the columns of $C$ form a nonnegative, orthonormal basis of the range of $P$, and these basis vectors therefore have disjoint supports.
\end{cor}
\begin{proof}
With reference to \eqref{nonneg projection matrix}, one has
\[
P = P^T \Longrightarrow PA = P^TA \Longleftrightarrow (A+B)C^TA = C(A^T+B^T)A \Longleftrightarrow A+B = C.
\]
Thus by \eqref{nonneg projection matrix} one has $P = CC^T$, and $C^TC = I$.
Since nonnegative vectors can only be orthogonal if they have disjoint supports, the columns of $C$ have this property.

Finally, recall that a projection matrix is symmetric if and only if it corresponds to an orthogonal projection.
\end{proof}

One may easily extend this to orthogonal projection operators, as follows.
\begin{prop}\label{NnPositiveToJConfig}
  Assume that a given orthogonal projection $P_S$ with range $S\subseteq {\mathbb{S}}^n$ satisfies
   $P_S(\mathcal{N}^n)\subseteq \mathcal{N}^n$.
   Then:
   \begin{enumerate}
   \item
    $S$ has a basis of nonnegative matrices with disjoint supports.
  \item
   If, in addition,  $S$ contains the all ones matrix $J$, then it is a partition subspace.
  \item
    If, in addition to the condition in item 2), $P_S({\mathbb{S}}^n_+)\subseteq {\mathbb{S}}^n_+$ and $S$ contains the identity matrix, then $S$ is a Jordan configuration.
  \end{enumerate}
  \end{prop}
\begin{proof}
Since $P_S$ is self-adjoint, we may write it as a symmetric matrix, say $M_{P_S}$, with respect to the standard orthonormal basis of $\mathbb{S}^n$.
For a $X \in \mathbb{S}^n$, we define the vector $\svec{(X)}
\in {\mathbb{R}}^{\frac{1}{2} n(n+1)}$  as $$ \svec{(X)} = \left(
X_{11},\sqrt{2} X_{21} ,\ldots,\sqrt{2} X_{n1},X_{22},\sqrt{2}
X_{32},\ldots,\sqrt{2} X_{n2}, \ldots,X_{nn} \right)^T. $$
Thus $\svec{(X)}$ gives the coordinates of $X$ in the standard orthonormal basis of $\mathbb{S}^n$.
One therefore has
\[
\svec{\left({P_S}(X)\right)} = M_{P_S}\cdot\svec{(X)} \quad\quad \forall X \in \mathbb{S}^n.
\]
Choosing $\svec{(X)}$ as the standard unit vectors in ${\mathbb{R}}^{\frac{1}{2} n(n+1)}$ makes it clear that $M_{P_S} \in \mathcal{N}^{\frac{1}{2}n(n+1)}$.
Thus the first claim now follows from Corollary \ref{cor:nonnegative projection}, namely that $S$ has a basis of nonnegative matrices with pairwise disjoint supports.
If $S$ contains the all-ones matrix $J$, then it must hold that these basis matrices are 0/1, proving the second claim.

Finally, to prove the third claim, we recall that $S$ unital and $P_S({\mathbb{S}}^n_+)\subseteq {\mathbb{S}}^n_+$ implies that $S$ is a Euclidean Jordan algebra.
Since it has a 0/1 basis, it is in fact a Jordan configuration if we also assume $I \in S$.

\end{proof}

The last proposition shows that partition subspaces are closely related to nonnegative projections.

The question remains if there exists an orthogonal projection $P_S:{\mathbb{S}}^n \rightarrow {\mathbb{S}}^n$ with range $S\subseteq {\mathbb{S}}^n$ that satisfies
   $P_S({\mathcal{D}}^n)\subseteq \mathcal{D}^n$, but \underline{not} $P_S(\mathcal{N}^n)\subseteq \mathcal{N}^n$.

   \begin{prop}\label{IdentityImpliesNPos}
     Let $P_S$ be an orthogonal projection with range $S$ that satisfies $P_S({\mathcal{D}}^n)\subseteq \mathcal{D}^n$. If $I\in S$, then $P_S(\mathcal{N}^n)\subseteq \mathcal{N}^n$.
   \end{prop}
   \begin{proof}
     Let $X\in \mathcal{N}^n$. Since $X\in \left({\mathcal{D}}^n\right)^* = \mathcal{N}^n + {\mathbb{S}}^n_+$ and $P_S(\left({\mathcal{D}}^n\right)^*)\subseteq \left({\mathcal{D}}^n\right)^*$, we know that the diagonal entries of $P_S(X)$ are nonnegative. To see the same for the off-diagonal entries, let $r$ be the spectral radius of $X$. Then $X+rI \in {\mathcal{D}}^n$, which implies that $P_S(X+rI) \in {\mathcal{D}}^n$ has nonnegative off-diagonal entries. Since $P_S(I)=I$ we have $P_S(X) = P_S(X+rI)-rI$, thus showing that $P_S(X)\in \mathcal{N}^n$.
   \end{proof}

    Hence all admissible subspaces that contain $J$ and $I$ are automatically Jordan configurations for conic problems over the doubly nonnegative cone, if they are ${\mathbb{S}}^n_+$-positive, by the last two propositions. It is still an open problem whether ${\mathcal{D}}^n$-positive implies ${\mathbb{S}}^n_+$-positive.

 % \begin{rem}
%    The additional restriction to $0/1$ bases is not w.l.o.g., for example consider the subspace spanned by a single idempotent matrix $A\geq 0$ with multiple different entries. This subspace is closed under squaring $(\lambda A)^2=\lambda^2 A$, and it has a unit $e=A$. Such an $A$ exists, e.g.
%    \begin{equation*}
%      A=\frac{1}{2}\begin{pmatrix}
%        1-\frac{1}{\sqrt{2}} & \frac{1}{\sqrt{2}} \\
%        \frac{1}{\sqrt{2}} & 1+\frac{1}{\sqrt{2}}
%      \end{pmatrix}.
%    \end{equation*}
%
%    So $S=\mathrm{span}\{A\}$ is a unital Jordan algebra, with basis $\{A\}$. Since $A\geq 0$, a ${\mathbb{S}}^n_+$-positive projection to $S$ is ${\mathcal{D}}^n$-positive, but $S$ does not have a $0/1$-basis.
%    Since coherent algebras are Jordan algebras, this shows that a restriction to the subspace having a $0/1$-basis or to the subspace being a coherent algebra is a truly stronger condition
%     than is needed by Theorem \ref{EquivalentCSICs}.
%  \end{rem}

\section{Reducing the semidefinite relaxation of the quadratic assignment problem}
\label{sec:QAPreduce}
 A semidefinite programming relaxation for $\mathrm{QAP}(A,B)$ (see \eqref{def:qapAB}), due to Zhao, Karisch, Rendl and Wolkowicz \cite{zhao1998semidefinite}, is
    \begin{align}\label{QAPSDP}
       \min\enspace & \langle B\otimes A ,Y\rangle\\
      \mathrm{s.t.}\enspace & \langle I_n\otimes E_{jj},Y\rangle=1 \text{ for }j\in [n],\nonumber\\
      & \langle E_{jj}\otimes I_n,Y\rangle=1 \text{ for }j\in [n],\nonumber\\
      & \langle I_n\otimes (J_n-I_n)+(J_n-I_n)\otimes I_n,Y\rangle =0,\nonumber \\
      & \langle J_{n^2},Y\rangle = n^2,\nonumber\\
      & Y\in D^{n^2},\nonumber
    \end{align}
    where $A,B\in {\mathbb{S}}^n$. We refer to \cite{povh2009copositive} for more details of this relaxation.

    First, we have to transform this program into the form on the right side of equation \eqref{def:conic opt problem}. We get a feasible solution $X_0$ by forming the outer product of a vectorized permutation-matrix, for example we can set
    \begin{equation*}
      X_0=\mathrm{vec}(I_n)\mathrm{vec}(I_n)^T.
    \end{equation*}
    We get the space $\mathcal{L}$, as seen earlier, by
    \begin{equation*}
      \mathcal{L}=\{X\in {\mathbb{S}}^{n^2}\enspace\vert\enspace\langle A_i,X\rangle=0 \quad\forall i\in [m]\},
    \end{equation*}
    where
    \begin{equation*}
      \{A_i\}_{i\in [m]} = \{J_{n^2}, I_n\otimes (J_n-I_n)+(J_n-I_n)\otimes I_n, I_n\otimes E_{jj} \mbox{ and } E_{jj}\otimes I_n \; (j\in [n])\}
    \end{equation*}
    are the data-matrices of the constraints of the SDP relaxation \eqref{QAPSDP}. Accordingly the orthogonal complement is exactly $\mathcal{L}^\perp = \mathrm{span}\{A_1,\ldots,A_m\}$.

\begin{thm}\label{NNProj}
  Any admissible subspace, say $S \subset \mathbb{S}^n$, for the QAP relaxation \eqref{QAPSDP} with $n>2$, has a basis of nonnegative matrices
  with disjoint supports, provided that it contains the identity matrix or $P_S(\mathcal{N}^n) \subseteq \mathcal{N}^n$.
  If both $P_S(\mathbb{S}^n_+) \subseteq \mathbb{S}^n_+$ and $S$ is unital, then $S$ is a Jordan configuration.
\end{thm}
\begin{proof}
Let $S$ be an admissible subspace for the QAP relaxation \eqref{QAPSDP} with $n>2$. The first claim of the theorem is an immediate consequence of Propositions \ref{NnPositiveToJConfig} and \ref{IdentityImpliesNPos}, since  $P_S(\mathbb{S}^n_+) \subseteq \mathbb{S}^n_+$ implies $I\in S$, as we will see below.

If we further assume $P_S(\mathbb{S}^n_+) \subseteq \mathbb{S}^n_+$ and $S$ unital,
then $S$ is closed under taking squares, i.e.\ it is a Jordan sub-algebra of $\mathbb{S}^n$ \cite[Lemma 5.2.2]{permenter2017reduction}.
Thus $S$ contains $X_{0,\mathcal{L}^\perp}$ and its square, which we will now calculate.
In \cite{brosch2020toric} the authors project $X_0=\mathrm{vec}(I_n)\mathrm{vec}(I_n)^T$ onto $\mathcal{L}^\perp$, the span of the constraint matrices:
\begin{equation*}
       X_{0,\mathcal{L}^\perp} = \frac{1}{n^2-n}(J_{n^2} - I_n\otimes J_n - J_n\otimes I_n) + \frac{1}{n-1} I_{n^2}.
     \end{equation*}

%we first notice only two of them have nonzero entries outside of the diagonal,
%       the all one matrix $J_{n^2}$, and the matrix  $I_n\otimes (J_n-I_n)+(J_n-I_n)\otimes I_n$,
%       which we will call $T$ from now on. The matrices $I_n\otimes E_{jj}$ for $j=1,\ldots,n$ sum to the identity matrix $I_{n^2}$,
%        meaning we can easily find an orthogonal basis of the off diagonal part of $\mathcal{L}^\perp$:
%     \begin{align*}
%       B_1 &= T, \\
%       B_2 &= J_{n^2} - I_{n^2} - T.
%     \end{align*}
%     Since $\langle T,X_0\rangle =0$ and $\langle J_{n^2},X_0\rangle = n^2$ we get
%     \begin{equation*}
%       \langle B_2, X_0\rangle = \langle J_{n^2},X_0\rangle - \langle I_{n^2},X_0\rangle = n^2 - n.
%     \end{equation*}
%     Hence the off-diagonal part of $X_{0,\mathcal{L}^\perp}$ is the matrix
%     \begin{equation*}
%       \frac{\langle B_2,X_0\rangle}{\langle B_2,B_2\rangle} B_2 = \frac{n^2-n}{n^4-n^2 - 2n(n^2-n)}B_2=\frac{1}{n^2-n}B_2.
%     \end{equation*}
%     The diagonal part of $X_{0,\mathcal{L}^\perp}$ is the matrix $\frac{1}{n} I_{n^2}$, since
%     \begin{align*}
%       \langle E_{jj}\otimes I_n, \frac{1}{n} I_{n^2} - X_0\rangle &= \langle E_{jj}\otimes I_n, \frac{1}{n} I_{n^2}\rangle -\langle E_{jj}\otimes I_n, X_0\rangle = 1-1=0,
%     \end{align*}
%     and analogously $\langle I_n\otimes E_{jj} , \frac{1}{n} I_{n^2} - X_0\rangle = 0$. Combining the two parts we see
%     \begin{equation*}
%       X_{0,\mathcal{L}^\perp} = \frac{1}{n^2-n}B_2 + \frac{1}{n} I_{n^2}.
%     \end{equation*}
 Straightforward calculation now yields
  \begin{align*}
    X_{0,\mathcal{L}^\perp}^2 & = \frac{n^2-2n+2}{n^2(n-1)^2}J_{n^2}-\frac{1}{n(n-1)^2}(I_n\otimes J_n+J_n\otimes I_n)+\frac{1}{(n-1)^2}I, \\
    X_{0,\mathcal{L}^\perp}^4      & = \frac{1}{(n-1)^2}X_{0,\mathcal{L}^\perp}^2 + \frac{n-2}{n(n-1)^2}J_{n^2}.
  \end{align*}
  Thus $S$ contains the all-ones matrix if $n > 2$, since
  \[
  \frac{n-2}{n(n-1)^2}J_{n^2} = X_{0,\mathcal{L}^\perp}^4- \frac{1}{(n-1)^2}X_{0,\mathcal{L}^\perp}^2,
  \]
   and the right-hand-side terms both belong to
  $S$.
  By {\color{black}Proposition \ref{NnPositiveToJConfig}}, $S$ must therefore have a $0/1$ basis, i.e.\ it must be a partition subspace.
  To show that it is in fact a Jordan configuration, we only need to show still that it contains the identity matrix.
  To this end, it suffices to note that all the diagonal entries of $ X_{0,\mathcal{L}^\perp}^2$ are the same, and different from the off-diagonal entries.
  Since $S$ has a $0/1$ basis, it must therefore contain the identity.
  \end{proof}

The important practical implication of this theorem is that the optimal admissible Jordan configuration $S$  of the QAP relaxation \eqref{QAPSDP}  may be computed using
Algorithm \ref{alg:Partitionrandomized}. The resulting reduction  is at least as good as the known ones from the literature, as we now show.

  %  \begin{thm}
%      The optimal unital, admissible partition subspace for the QAP relaxation \eqref{QAPSDP} is a \emph{Jordan-configuration}.
%    \end{thm}
%    \begin{proof}
%    The proof uses the fact that Algorithm \ref{alg:Partitionrandomized} returns the optimal unital, admissible partition subspace.
%      All data matrices for the QAP relaxation \eqref{def:qapAB} are symmetric,
%      so the initial partition in Algorithm \ref{alg:Partitionrandomized} is symmetric. Squaring a symmetric matrix results in another symmetric matrix, and Algorithm \ref{alg:Partitionrandomized}
%      projects  onto a subspce of symmetric matrices, so the resulting partition is symmetric as well.
%      As seen earlier, $XY+YX = (X+Y)^2-X^2-Y^2$, so the partition is closed under the product $XY+YX$. The final property follows from Lemma \ref{IInPart}.
%    \end{proof}

    \begin{cor}
      This symmetry reduction of the QAP relaxation \eqref{def:qapAB} via Algorithm \ref{alg:Partitionrandomized}  is at least as good as both the group symmetry reduction (see \cite{de2010exploiting,de2012improved})
       and the reduction to the coherent algebra containing the data matrices of the program (via the Weisfeiler-Leman algorithm \cite{WL}).
    \end{cor}
    \begin{proof}
      The symmetric part of a coherent configuration is a Jordan configuration, and the partition given by the orbitals of a group leaving the program invariant is a coherent configuration.
    \end{proof}

\subsection{Results of reductions of QAPLib problems}

    In practice the (partition) Jordan reduction is not much stronger than group symmetry reduction,
     and reduction to the smallest coherent algebra containing the data matrices.
     When comparing reductions for data from QAPLib \cite{burkard1997qaplib},
     only a single reduction (esc16f), of the ones that were symmetry reduced before, was stronger,
     the others were exactly the same as reported in \cite{de2010exploiting}, where the reduction was done
     using group symmetry. But we managed to  reduce some larger instances for the first time.
     We also do gain a large speed up in determining the reduction, since we avoid having to determine the automorphism groups of matrices.
     In Table \ref{QAPLibTable} we give the dimension of the smallest admissible partition subspace for each problem (for which we determined a reduction),
      the original number of variables of the problem, and the time needed for the reduction.{\color{black} In Table \ref{QAPLibResultsTable} we show the time needed to block-diagonalize (using the algorithm described in \cite{murota2010numerical}) and solve these problems afterwards (if the dimension of the admissible subspace was $\leq 3000$), as well as the resulting bounds. The optimal value of most of these is known exactly, which we give in the last row, taken from \url{http://anjos.mgi.polymtl.ca/qaplib/inst.html}.}
      
{\scriptsize
\npdecimalsign{.}
\nprounddigits{3}

\begin{table}[h!]
  \centering
\scriptsize
    \begin{tabular}{|c|l|l|n{2}{3}|}
    \hline
QAP & Dim. & Red. dim. & \text{Jordan red. (s)}  \\ \hline
chr18b & 52650 & 14742 & 0.3579816 \\
esc16a & 32896 & 150 & 0.1621989 \\
esc16b & 32896 & 155 & 0.191899599 \\
esc16c & 32896 & 405 & 0.194103 \\
esc16d & 32896 & 405 & 0.170898099 \\
esc16e & 32896 & 135 & 0.165644201 \\
esc16f & 32896 & 3 & 0.0999377 \\
esc16g & 32896 & 230 & 0.207212699 \\
esc16h & 32896 & 90 & 0.130286299 \\
esc16i & 32896 & 280 & 0.2538829 \\
esc16j & 32896 & 150 & 0.2137541 \\
esc32a & 524800 & 2112 & 3.826268 \\
esc32b & 524800 & 96 & 3.306410299 \\
esc32c & 524800 & 366 & 3.228223 \\
esc32d & 524800 & 342 & 3.0969752 \\
esc32e & 524800 & 120 & 2.884582401 \\
esc32g & 524800 & 180 & 2.8576561 \\
esc32h & 524800 & 666 & 3.0512814 \\
esc64a & 8390656 & 679 & 57.581179701 \\
kra32 & 524800 & 28752 & 3.099143199 \\
nug12 & 10440 & 2952 & 0.077246099 \\
nug15 & 25425 & 7425 & 0.1520472 \\
nug16b & 32896 & 4704 & 0.147158899 \\
nug20 & 80200 & 21000 & 0.8191775 \\
nug21 & 97461 & 27783 & 0.473928599 \\
nug22 & 117370 & 29766 & 0.7572002 \\
nug24 & 166176 & 41760 & 1.010178699 \\
nug25 & 195625 & 28675 & 1.132266399 \\
nug27 & 266085 & 75087 & 1.489256 \\
nug28 & 307720 & 78792 & 1.865216699 \\
scr12 & 10440 & 2952 & 0.0565526 \\
scr15 & 25425 & 13275 & 0.146641401 \\
tai64c & 8390656 & 75 & 55.8387828 \\
tho30 & 405450 & 112950 & 2.9268635 \\
tho40 & 1280800 & 333600 & 9.476511399 \\
wil50 & 3126250 & 813750 & 27.1227091 \\\hline

    \end{tabular}
  \caption{Results for numerical symmetry reduction of QAPLib problems using Algorithm \ref{alg:Partitionrandomized}.}\label{QAPLibTable}
\end{table}

\npdecimalsign{.}
\nprounddigits{3}

\begin{table}[h!]
  \centering
\scriptsize
    \begin{tabular}{|c|n{4}{3}|n{2}{3}|c|n{7}{3}|l|}
    \hline
QAP & \text{block-diag. (s)} & \text{solve (s)} & blocks (size $\times$ mult) & \text{optimal value \eqref{QAPSDP}} & \text{QAP optimum}  \\ \hline
esc16a & 0.8839299 & 0.2292924 & $6\times5$,$3\times5$,$1\times15$, & 63.28541990404178 & 68\\
esc16b & 0.749267 & 0.284997 & $7\times5$,$1\times15$, & 289.9988558060066 & 292\\
esc16c & 2.765734201 & 0.7588051 & $12\times5$,$1\times15$, & 153.9988307193536 & 160\\
esc16d & 2.7136434 & 0.3730593 & $12\times5$,$1\times15$, & 12.99999863716645 & 16\\
esc16e & 0.753351399 & 0.158655899 & $6\times5$,$2\times5$,$1\times15$, & 26.336797695483035 & 28\\
esc16f & 0.0402993 & 0.0479962 & $1\times3$, & 0.0 & 0\\
esc16g & 1.1086753 & 0.2171645 & $9\times5$,$1\times5$, & 24.740307071561457 & 26\\
esc16h & 0.414226699 & 0.125240401 & $5\times5$,$1\times15$, & 976.2279051431194 & 996\\
esc16i & 1.572374401 & 0.296006 & $10\times5$,$1\times5$, & 11.374895676633026 & 14\\
esc16j & 0.6901062 & 0.167068 & $7\times5$,$1\times10$, & 7.794218632983682 & 8\\
esc32a & 482.027258499 & 23.958262401 & $26\times6$,$1\times6$, & 103.31959358725226 & 130\\
esc32b & 11.501883599 & 0.040621 & $2\times24$,$1\times24$, & 131.8828076108162 & 168\\
esc32c & 51.1459494 & 0.2958341 & $10\times6$,$1\times36$, & 615.1780270156789 & 642\\
esc32d & 56.076293301 & 0.212864199 & $9\times6$,$2\times12$,$1\times36$, & 190.22703493191395 & 200\\
esc32e & 11.436337901 & 0.0542997 & $5\times6$,$1\times30$, & 1.8999999468859592 & 2\\
esc32g & 14.941279001 & 0.0960032 & $7\times6$,$1\times12$, & 5.833332018693771 & 6\\
esc32h & 114.9430212 & 1.134790199 & $14\times6$,$1\times36$, & 424.39840742624796 & 438\\
esc64a & 1985.9167665 & 0.885204601 & $13\times7$,$2\times7$,$1\times21$, & 97.7497923767676 & 116\\
nug12 & 12.883667299 & 80.018891801 & $48\times2$,$24\times2$, & 567.9696928304795 & 578\\
scr12 & 12.894117901 & 83.330367401 & $48\times2$,$24\times2$, & 31409.996810456192 & 31410\\
tai64c & 182.909332799 & 0.1525016 & $2\times15$,$1\times30$, & 1811366.4813202813 & $\geq$1855928\\\hline

    \end{tabular}
  \caption{Details on solving \eqref{QAPSDP} for QAPLib instances via block-diagonalization.\label{QAPLibResultsTable}}
\end{table}
    }
\section{Reducing the $\vartheta'$-function}
\label{sec:ThetaReduc}
{\color{black}
The $\vartheta'$-function of a Graph $G = (V,E)$, as given in \eqref{thetaPrime}, is a doubly nonnegative semidefinite program of size $n\coloneqq\vert V\vert$. Here we can say a bit less about admissible subspaces in the general case. As seen in Theorem \ref{thm:equivalent CSICs}, every admissible subspace needs to contain $C_\mathcal{L}$ and $X_{0,\mathcal{L}^\perp}$. Here it is straightforward to see that $C_\mathcal{L}$ is exactly the adjacency matrix of the complementary graph $\overline{G}$, and $X_{0,\mathcal{L}^\perp} = \frac{1}{n}I_n$. Thus we know at least that $S$ contains the identity. This implies that every admissible subspace for the $\vartheta'$-function has a basis of nonnegative matrices with disjoint supports, by Propositions \ref{NnPositiveToJConfig} and \ref{IdentityImpliesNPos}.

But an admissible subspace here does not necessarily need to contain the all-one matrix $J_n$. One obtains an easy example by $G = ([n], \{\{i,j\} \text{ if } i > m \text{ or } j > m\}$ for $n > m > 0$. It is easy to check that an admissible subspace for this problem is of the form
$$\begin{pmatrix}
    a & b & b &  &  &  &  \\
    b & a & b &  &  &  &  \\
    b & b & a &  &  &  &  \\
     &  &  & c &  &  &  \\
     &  &  &  & c &  &  \\
     &  &  &  &  & c &  \\
     &  &  &  &  &  & c
  \end{pmatrix},$$
  here shown for $n = 7$ and $m = 3$. While this case is not too interesting, this shows that the Jordan-Reduction can be better than a group-symmetry reduction, which would result in the five-dimensional subspace given by
  $$\begin{pmatrix}
    a & b & b & d & d & d & d \\
    b & a & b & d & d & d & d \\
    b & b & a & d & d & d & d \\
    d & d & d & c & e & e & e \\
    d & d & d & e & c & e & e \\
    d & d & d & e & e & c & e \\
    d & d & d & e & e & e & c
  \end{pmatrix}.$$
  Do note though that in this case the three additional variables will be eliminated by the constraints of the SDP soon after.

}
\subsection{The $\vartheta'$-function of \ER graphs}\label{ThetaER}

{\color{black}

Let $q$ be an odd prime, and let $V = \mathrm{GF}(q)^3$ be a three dimensional vector space over the finite field of order $q$. The set of one dimensional subspaces, i.e. the projective plane, of $V$ is denoted by $\mathrm{PG}(2,q)$. There are $q^2+q+1$ such subspaces, which are the vertices of the \ER graph $\mathrm{ER}(q)$. Two vertices are adjacent if they are distinct and orthogonal, i.e. for two representing vectors $x$ and $y$ we have $x^Ty=0$. We are interested in the size of a maximum stable set of these graphs, specifically upper bounds for this value.

In \cite{Godsil2008EVBoundER} the authors derive the upper bound
\begin{equation}\label{GodsilEVBound}
  \frac{\sqrt{q}+\sqrt{q+4(q+1)\frac{q+\sqrt{q}+1}{q^2+q+1}}}{2\frac{q+\sqrt{q}+1}{q^2+q+1}},
\end{equation}
which was shown to be at most as good as the $\vartheta'$-function in \cite{klerk2009ERGraphs}.

The $\vartheta'$-function of $\mathrm{ER}(q)$ is a doubly nonnegative semidefinite program of size $q^2+q+1$. Without further reductions one can practically solve this for up to $q=17$. In \cite{klerk2009ERGraphs} the authors reduced the problem size enough to solve it for up to $q=31$, and in this paper we managed to solve it for up to $q=97$.

We applied the reduction algorithm, numerically block-diagonalized (more on that next section) and solved the resulting problems for all primes from $q=3$ to $97$, as shown in Tables \ref{ERReductionTable} and \ref{ERBoundsTable}. Interestingly, the reduced block sizes always are one block of size $3\times 3$, and $\lceil \frac{q}{2}\rceil$ blocks of size $2\times 2$, i.e. the problem nearly reduces to a second order cone problem. By comparison, the problem was reduced to SDPs of matrix size $2q+11$ in \cite{klerk2009ERGraphs}.
}

\npdecimalsign{.}
\nprounddigits{3}

\begin{table}[h!]
  \centering
\scriptsize
    \begin{tabular}{|l|l|n{3}{3}|n{4}{3}|c|}
    \hline
$q$  & $q^2+q+1$ & \text{Jordan red. (s)} & \text{block diag. (s)} & blocks (size $\times$ mult.)   \\ \hline
3  & 13   & 0.000558001   & 0.0014838      & $3\times 1,2\times 2$  \\
5  & 31   & 0.002265899   & 0.0058976      & $3\times 1,2\times 3$  \\
7  & 57   & 0.0070165     & 0.011389401    & $3\times 1,2\times 4$  \\
11 & 133  & 0.0666903     & 0.0330895      & $3\times 1,2\times 6$  \\
13 & 183  & 0.0916287     & 0.050578501    & $3\times 1,2\times 7$  \\
17 & 307  & 0.240929799   & 0.169581799    & $3\times 1,2\times 9$  \\
19 & 381  & 0.4205891     & 0.3032458      & $3\times 1,2\times 10$ \\
23 & 553  & 1.01934       & 0.7226254      & $3\times 1,2\times 12$ \\
29 & 871  & 2.455367801   & 2.392133399    & $3\times 1,2\times 15$ \\
31 & 993  & 3.398327099   & 3.664145299    & $3\times 1,2\times 16$ \\
37 & 1407 & 7.744959701   & 9.0679678      & $3\times 1,2\times 19$ \\
41 & 1723 & 13.038789501  & 16.3578363     & $3\times 1,2\times 21$ \\
43 & 1893 & 14.5333772    & 21.507566401   & $3\times 1,2\times 22$ \\
47 & 2257 & 19.9104259    & 36.710783      & $3\times 1,2\times 24$ \\
53 & 2863 & 33.2706014    & 72.051950401   & $3\times 1,2\times 27$ \\
59 & 3541 & 51.462697299  & 140.1193346    & $3\times 1,2\times 30$ \\
61 & 3783 & 54.7143953    & 166.2674879    & $3\times 1,2\times 31$ \\
67 & 4557 & 78.579218401  & 332.437736999  & $3\times 1,2\times 34$ \\
71 & 5113 & 115.303399801 & 487.1617033    & $3\times 1,2\times 36$ \\
73 & 5403 & 118.0582613   & 545.4981241    & $3\times 1,2\times 37$ \\
79 & 6321 & 179.0841249   & 886.0648295    & $3\times 1,2\times 40$ \\
83 & 6973 & 215.9833152   & 1336.901259101 & $3\times 1,2\times 42$ \\
89 & 8011 & 293.9469087   & 1931.7226019   & $3\times 1,2\times 45$ \\
97 & 9507 & 434.3407695   & 2912.8402779   & $3\times 1,2\times 49$ \\\hline

    \end{tabular}
  \caption{Results of the numerical symmetry reduction of the Theta' function of Erdos-Renyi graphs.}\label{ERReductionTable}
\end{table}

\begin{table}[h!]
  \centering
\scriptsize
    \begin{tabular}{|l|n{1}{3}|n{3}{3}|n{3}{3}|}
    \hline
$q$  & \text{solve time (s)} & \text{$\vartheta'(\mathrm{ER}(q))$} & \text{EV bound \eqref{GodsilEVBound}}   \\ \hline
3  & 0.0022267    & 5.0000000765254375 & 5.559861885833123\\
5  & 0.002530101  & 10.066926506194214 & 10.55584699685833\\
7  & 0.002468799  & 15.743402859021042 & 16.72680909967116\\
11 & 0.002723201  & 31.08770429354092  & 32.050620288667396\\
13 & 0.0027327    & 40.50939213388844  & 41.024888544944595\\
17 & 0.0035806    & 60.22099922159293  & 61.29126346783692\\
19 & 0.0040003    & 71.3009523623285   & 72.49317178390179\\
23 & 0.0040835    & 96.24003796685733  & 96.85844748525206\\
29 & 0.0056097    & 136.97844019579597 & 137.90950659356875\\
31 & 0.007177601  & 151.7024311425505  & 152.70736432626447\\
37 & 0.006835101  & 199.2688507099146  & 200.2030992687015\\
41 & 0.0085739    & 233.39019647298073 & 234.3118960646621\\
43 & 0.009387901  & 250.91677057318012 & 252.06324483077262\\
47 & 0.0112386    & 287.77164146620225 & 288.90727694602305\\
53 & 0.0128003    & 346.6261041746673  & 347.388336922548\\
59 & 0.014773     & 408.54845511033363 & 409.5338769450585\\
61 & 0.0149637    & 430.21947678719636 & 431.03021084515547\\
67 & 0.0201625    & 496.4378437355285  & 497.77513984846314\\
71 & 0.0189927    & 543.127953133481   & 544.0954929733048\\
73 & 0.020974099  & 566.9145884004363  & 567.7871911965542\\
79 & 0.9450113    & 639.6442218693393  & 640.9316087535851\\
83 & 1.1105472    & 690.583316369887   & 691.3752759274242\\
89 & 0.1145007    & 768.4692537601204  & 769.4814551650646\\
97 & 0.108091901  & 877.075250044101   & 878.0269560540745\\\hline

    \end{tabular}
  \caption{The resulting bounds for the stable set number of Erdos-Renyi graphs.}\label{ERBoundsTable}
\end{table}

    \section{The Julia package}\label{JuliaPackageSection}
    \label{sec:JuliaPackage}
    {\color{black}

    We provide a package "SDPSymmetryReduction.jl" as part of the Julia registry, available at \url{https://github.com/DanielBrosch/SDPSymmetryReduction.jl}. We provide functions to both find an optimal admissible partition subspace for a given SDP, as well as to consequently block-diagonalize it.

    To enter a semidefinite programm one has to provide (potentially sparse) vectors and matrices $C\in \mathbb{R}^{n^2},A\in \mathbb{R}^{m\times n^2}$ and $b\in\mathbb{R}^m$ as in \eqref{def:conic opt problem} (vectorizing the variable $X$). We provide examples for how one can approach this for both the $\vartheta'$-function of a given graph and for the SDP-bound \eqref{QAPSDP}. Since we return a partition based symmetry reduction, it is not necessary to give entry-wise nonnegativity constraints, as long as one remembers to use nonnegative variables in the final, reduced SDP (see also the example in \ref{CodeExample}).

    \subsubsection*{Determining an admissible subspace}
The function `$\mathrm{admPartSubspace}$` determines an optimal admissible partition subspace for the problem, by Algorithm \ref{alg:Partitionrandomized}. This is done using a randomized Jordan-reduction algorithm, and it returns a Jordan algebra. SDPs can be restricted to such a subspace without changing their optimal value.

Given $C$,$A$ and $b$, $\mathrm{admPartSubspace}(C,a,b)$ returns a $\mathrm{Partition}$ $P$ with $P.n$ giving the number of parts of the partition, and $P.P$ returning an integer valued matrix with entries $1,\ldots,n$ defining the partition.

For example, let $C$, $A$ and $b$ define the $\vartheta'$-function of the cycle graph $C_5$. If we label the vertices such that its adjacency matrix is
$$\begin{pmatrix}
    0 & 1 & 0 & 0 & 1 \\
    1 & 0 & 1 & 0 & 0 \\
    0 & 1 & 0 & 1 & 0 \\
    0 & 0 & 1 & 0 & 1 \\
    1 & 0 & 0 & 1 & 0
  \end{pmatrix},$$
then calling $\mathrm{admPartSubspace}(C,a,b)$ returns the partition $P$ with $P.n=3$ and
$$P.P = \begin{pmatrix}
    1 & 2 & 3 & 3 & 2 \\
    2 & 1 & 2 & 3 & 3 \\
    3 & 2 & 1 & 2 & 3 \\
    3 & 3 & 2 & 1 & 2 \\
    2 & 3 & 3 & 2 & 1
  \end{pmatrix},$$
i.e. we can restrict the feasible set to the three-dimensional subspace given by $P$.

\subsubsection*{Block-diagonalizing a Jordan-algebra}
The function `$\mathrm{blockDiagonalize}$` determines a block-diagonalization of a (Jordan)-algebra given by a partition $P$ using a randomized algorithm. It implements the Algorithm from \cite{murota2010numerical} (see also \cite{de2011numerical}). To our knowledge this is the first implementation available to the public.

$\mathrm{blockDiagonalize}(P)$ returns a real block-diagonalization $blkd$, if it exists, otherwise `nothing`.
\begin{itemize}
  \item $blkd.blkSizes$ returns an integer array of the sizes of the blocks.
  \item $blkd.blks$ returns an array of length $P.n$ containing arrays of (real) matrices of sizes $blkd.blkSizes$. I.e. $blkd.blks[i]$ is the image of the basis element given by the $0/1$-matrix with a one in the positions where $P.P$ is $i$.
\end{itemize}

$\mathrm{blockDiagonalize}(P; complex = true)$ returns the same, but with complex valued matrices, and should be used if no real block-diagonalization was found. To use the complex matrices in practice, remember that a Hermitian matrix $Y$ is positive semidefinite if and only if $$\begin{pmatrix}
                        \mathrm{real}(Y) & -\mathrm{imag}(Y) \\
                        \mathrm{imag}(Y)& \mathrm{real}(Y)
                      \end{pmatrix}$$ is positive semidefinite.

Continuing the example of reducing $\vartheta'(C_5)$, $\mathrm{blockDiagonalize}(P)$ here returns a block-diagonalization $blkd$ with $blkd.blkSizes = [1,1,1]$ (i.e. three blocks of size $1\times 1$), and the image of the basis is
\begin{align*}
  blkd.blks \approx [& [1,1,1], \\
   & [-1.618, 0.618,2],\\
   & [0.618,-1.618,2]].
\end{align*}
This means that
$$
\begin{pmatrix}
    a & b & c & c & b \\
    b & a & b & c & c \\
    c & b & a & b & c \\
    c & c & b & a & b \\
    b & c & c & b & a
  \end{pmatrix} \succcurlyeq 0 \Leftrightarrow \left\{
    \begin{array}{@{} l @{}}
      a - 1.618b + 0.618c \geq 0,  \\
      a + 0.618b - 1.618\geq 0, \\
      a + 2b + 2c \geq 0,
    \end{array}\right.
$$
which allows us to rewrite $\vartheta'(C_5)$ as a linear program in the three variables $a,b,c$.

%    The main function of the library is "JordanReduction01", accepting an SDP as input, which implements Algorithm \ref{alg:Partitionrandomized}. It returns a partition $P$, with $P.n$ giving the total number of parts, and $P.P$ an integer matrix with entries from $1$ to $P.n$ defining the partition.
%
%    The second main function is "blockDiagonalize", which accepts a Partition, an epsilon, and a boolean to decide whether complex numbers should be used. It implements the Algorithm from \cite{murota2010numerical} (see also \cite{de2011numerical}). To our knowledge this is the first implementation available to the public. It returns an array of matrices for each $i=1,\ldots,P.n$ representing the image of the part given by the entry $i$ in $P.P$. If chosen to block-diagonalize with real numbers the algorithm can fail, returning nothing (as not every $*$-Algebra can be block-diagonalized over the reals). In that case one should run the function again in its complex-valued variant. Note that this is not a problem, since a hermitian $n\times n$-matrix $Y$ is positive semidefinite if and only if the real symmetric $2n\times 2n$-matrix $\begin{pmatrix}
%                        \mathrm{real}(Y) & -\mathrm{imag}(Y) \\
%                        \mathrm{imag}(Y)& \mathrm{real}(Y)
%                      \end{pmatrix}$ is positive semidefinite.
    }

\section{Concluding remarks}
We have extended the Jordan symmetry reduction method to the doubly nonnegative cone, and showed that for this cone that the restriction to admissible subspaces that are partition-spaces is not a strong requirement in Section \ref{sec:extenD}. We have shown that indeed the optimal admissible subspace semidefinite programming relaxation of the quadratic assignment problem is a coherent-configuration in Theorem \ref{NNProj} of Section \ref{sec:QAPreduce} under some weak requirements, and applied this to the symmetric instances of QAPLib. In Section \ref{sec:ThetaReduc} we have seen that the optimal admissible subspace of the $\vartheta'$-function can always be given by a nonnegative basis with disjoint supports, and applied this reduction to \ER-graph instances. Finally we describe the Julia package "SDPSymmetryReduction", available at \url{https://github.com/DanielBrosch/SDPSymmetryReduction.jl}, implementing the described algorithms in Section \ref{sec:JuliaPackage}. In \ref{CodeExample} we give a complete code example as to how one can use this package to calculate $\vartheta'(\mathrm{ER}(q))$.

%\subsection*{Acknowledgement}
%This project has received funding from the European Union’s Horizon 2020 research and innovation programme under the Marie Sk{\l}odowska-Curie
%grant agreement No 764759 (MINOA).

\appendix
\section{Example use of the software package}\label{CodeExample}
In this appendix we give a complete example how to setup the $\vartheta'$-function of $\mathrm{ER}(q)$, where $q=31$, and how to block-diagonalize, and then solve the reduced SDP with JuMP and Mosek.

    \lstset{language=Julia}

    \begin{lstlisting}
using SDPSymmetryReduction
using LinearAlgebra, SparseArrays
using JuMP, MosekTools

## Calculating the Theta'-function of Erdos-Renyi graphs
q = 31

# Generating the adjacency matrix of ER(q)
PG2q = vcat([[0, 0, 1]],
            [[0, 1, b] for b = 0:q-1],
            [[1, a, b] for a = 0:q-1 for b = 0:q-1])
Adj = [x' * y % q == 0 for x in PG2q, y in PG2q]
Adj[diagind(Adj)] .= 0

# Theta' SDP
N = length(PG2q) # = q^2+q+1
C = ones(N^2)
A = hcat(vec(Adj), vec(Matrix(I, N, N)))'
b = [0, 1]

# Find the optimal admissible subspace (= Jordan algebra)
P = admPartSubspace(C, A, b, true)

# Block-diagonalize the algebra
blkD = blockDiagonalize(P, true)

# Calculate the coefficients of the new SDP
PMat = hcat([sparse(vec(P.P .== i)) for i = 1:P.n]...)
newA = A * PMat
newB = b
newC = C' * PMat

# Solve with optimizer of choice
m = Model(Mosek.Optimizer)

# Initialize variables corresponding parts of the partition P
# >= 0 because the original SDP-matrices are entry-wise nonnegative
x = @variable(m, x[1:P.n] >= 0)

@constraint(m, newA * x .== newB)
@objective(m, Max, newC * x)

# Setup the block-diagonalized PSD-constraints
psdBlocks = sum(blkD.blks[i] .* x[i] for i = 1:P.n)
for blk in psdBlocks
    if size(blk, 1) > 1
        @constraint(m, blk in PSDCone())
    else
        @constraint(m, blk .>= 0)
    end
end

optimize!(m)

@show termination_status(m)
@show value(newC * x)

    \end{lstlisting}

{\footnotesize
\bibliographystyle{plain}

}

\end{document}